\documentclass[leqno,11pt]{article}
\usepackage{amssymb,amsmath,amsthm,amsfonts}
\usepackage{verbatim}
\allowdisplaybreaks

\usepackage[utf8]{inputenc}

\theoremstyle{plain}
\newtheorem{theorem}{Theorem}
\newtheorem{lemma}[theorem]{Lemma}
\newtheorem{corollary}[theorem]{Corollary}
\newtheorem{proposition}[theorem]{Proposition}

\theoremstyle{definition}
\newtheorem{definition}[theorem]{Definition}

\newtheorem{example}[theorem]{Example}

\newtheorem{remark}[theorem]{Remark}

\newcommand{\eps}{\epsilon}
\newcommand{\ad}{\text{ad}}

\newcommand{\R}{\mathbb R}

\newcommand{\G}{\mathbb G}

\newcommand{\Al}{f^{A_l}(x_0)}
\newcommand{\ai}{f^{A_i}(x)}
\newcommand{\Aj}{f^{A_j}(x_0)}
\newcommand{\Ai}{f^{A_i}(x_0)}
\newcommand{\bmlk}{\beta_m^{lk}}
\newcommand{\al}{f^{A_l}(x)}
\newcommand{\aj}{f^{A_j}(x)}
\newcommand{\Bj}{f^{B_j}(x_0)}
\newcommand{\bj}{f^{B_j}(x)}
\newcommand{\Bk}{f^{B_k}(x_0)}
\newcommand{\akij}{\alpha_k^{ij}}
\newcommand{\akji}{\alpha_k^{ji}}
\newcommand{\bmij}{\beta_m^{ij}}
\newcommand{\bk}{f^{B_k}(x)}
\newcommand{\pt}{\text{pt}}

\newcommand{\bmck}{\beta_m^{ck}}
\newcommand{\akcb}{\alpha_k^{cb}}
\newcommand{\akbc}{\alpha_k^{bc}}

\newcommand{\akab}{\alpha_k^{ab}}
\newcommand{\bmak}{\beta_m^{ak}}
\newcommand{\akca}{\alpha_k^{ca}}
\newcommand{\bmbk}{\beta_m^{bk}}
\newcommand{\skos}{\sum_{k=1}^s}
\newcommand{\smot}{\sum_{m=1}^t}

\newcommand{\akjd}{\alpha_k^{jd}}

\newcommand{\akcd}{\alpha_k^{cd}}
\newcommand{\akjc}{\alpha_k^{jc}}

\newcommand{\bmdk}{\beta_m^{dk}}
\newcommand{\ddcm}{\frac{d}{dC_m}}
\newcommand{\ddbk}{\frac{d}{dB_k}}

\numberwithin{theorem}{section} \numberwithin{equation}{section}

\title{A variant of Gromov's problem on H\"{o}lder equivalence  of Carnot groups}
\author{Derek Jung\footnote{Supported  by U.S. Department of Education GAANN fellowship P200A150319. \hfill\break {\it Key Words and Phrases:} sub-Riemannian geometry, Carnot groups, H\"older mappings, geometric measure theory, jet spaces, Gromov conjecture, unrectifiability \hfill\break {\it 2010 Mathematics Subject Classification:} Primary 53C17, 22E25; Secondary 53C23, 26B35, 49Q15,  58A20} \\ Department of Mathematics \\ University of Illinois at Urbana-Champaign \\ 1409 West Green St. \\ Urbana, IL 61801 \\ {\tt djjung2@illinois.edu}}
\date{\today}

\begin{document}

\raggedbottom

\maketitle

\begin{abstract}
It is unknown if there exists a locally $\alpha$-H\"older homeomorphism $f:\R^3\to \mathbb{H}^1$ for any $\frac{1}{2}< \alpha\le \frac{2}{3}$, although the identity map $\R^3\to \mathbb{H}^1$ is locally $\frac{1}{2}$-H\"older. More generally, Gromov asked: Given $k$ and a Carnot group $G$, for which $\alpha$ does there exist a locally $\alpha$-H\"older homeomorphism $f:\R^k\to G$? Here, we equip a Carnot group $G$ with  the Carnot-Carath\'eodory metric. In 2014, Balogh, Haj{\l}asz, and Wildrick considered a variant of this problem. These authors proved that if $k>n$, there does not exist an  injective, $(\frac{1}{2}+)$-H\"older mapping  $f:\R^k\to \mathbb{H}^n$  that is also locally Lipschitz as a mapping into $\R^{2n+1}$. For their proof, they use the fact that $\mathbb{H}^n$ is purely $k$-unrectifiable for $k>n$. In this paper, we will extend their result from the Heisenberg group to model filiform groups and Carnot groups of step at most three. We will now require that the Carnot group is  purely $k$-unrectifiable. The main key to our proof will be showing that  $(\frac{1}{2}+)$-H\"older maps $f:\mathbb{R}^k\to G$ that are  locally Lipschitz into Euclidean space, are weakly contact. Proving weak contactness in these two settings requires understanding the relationship between the  algebraic  and metric structures of the Carnot group. We will use coordinates of the first and second kind for Carnot groups.
\end{abstract}

\tableofcontents
 
\setcounter{section}{0}
\setcounter{subsection}{0}

\section{Introduction}

A Lie algebra $\mathfrak{g}$ is said to have an $r$-step \textbf{stratification} if
\[
\mathfrak{g}= \mathfrak{g}_1\oplus\mathfrak{g}_2\oplus \cdots \oplus \mathfrak{g}_r,
\]
where $\mathfrak{g}_1\subseteq \mathfrak{g}$ is a subspace, $\mathfrak{g}_{j+1} = [\mathfrak{g}_1,\mathfrak{g}_j]$ for all $j=1,\ldots , r-1$, and $[\mathfrak{g}, \mathfrak{g}_r] =0$. 
A \textbf{Carnot group}  is a connected, simply-connected, nilpotent Lie group with a stratified Lie algebra. 
If the Lie algebra of a Carnot group $G$ admits an $r$-step stratification, then we will say $G$ is step $r$. 
Each Carnot group can be identified with a Euclidean space equipped with a metric structure and a group operation arising from its Lie algebra structure.

It is natural to ask the following general question:
\begin{center}
When are two Carnot  groups equivalent?
\end{center}

In \cite{MDC:P}, Pansu proved that two Carnot groups are biLipschitz homeomorphic if and only if they are isomorphic. 
With the problem of biLipschitz equivalence  somewhat well-understood,  we can go on to ask when two Carnot groups are H\"older equivalent.

In \cite{CCS:G}, Gromov considered the problem of H\"older equivalence of Carnot groups:
If a Carnot group $G$ is identified with $\R^n$ equipped with a group operation, for which $\alpha$ does there exist a locally $\alpha$-H\"older homeomorphism $f:\R^n\to G$?
If such $\alpha$ exist, what is the supremum of the set of such $\alpha$? 
Here, we do not require any regularity of $f^{-1}$ beyond continuity.

Before we discuss past work on this problem, we will comment on the notation that will be used throughout this paper. 
We will simply write $\R^n$ to denote Euclidean space equipped with addition and the standard Euclidean metric.
\emph{We will write $(\R^n,\cdot)$ to denote a Carnot group equipped with coordinates of the first or second kind and with the Carnot-Carath\'eodory metric.}
When we equip a Carnot group with coordinates of the first or second kind, it is implied that we are taking coordinates with respect to a basis compatible with the stratification of its Lie algebra. 
We will introduce these two systems of coordinates and the Carnot-Carath\'eodory metric for  Carnot groups in section 2.
In section 3, we will discuss  coordinates of the second kind for a class of jet spaces: the model filiform groups.
We will begin section 4 by  looking at the geometry of Carnot groups of step at most three.

Nagel, Stein, and Wenger \cite[Proposition 1.1]{NSW:BAM} proved the existence of $\alpha$ as above: 
\begin{proposition}\label{holder}
Let $(\R^n,\cdot)$ be a step $r$ Carnot group.
Then $\text{id}:\R^n \to (\R^n,\cdot)$ is locally $\frac{1}{r}$-H\"older  
and $\text{id}:(\R^n,\cdot)\to \R^n$ is locally Lipschitz.
\end{proposition}

On the other hand,  Gromov \cite[Section 4]{CCS:G} used an  isoperimetric inequality for Carnot groups \cite{TDP:V} to prove  that if there exists a locally $\alpha$-H\"older homeomorphism $f:\R^n\to (\R^n,\cdot)$, then
\[
\alpha \le \frac{n-1}{Q-1}.
\]
Here, $Q$ denotes the Hausdorff dimension of $(\R^n,\cdot)$ with respect to its cc-metric.

Beyond these results, little is known about this problem. 
For example, in the case of the first Heisenberg group, the supremum of $\alpha$ for which there exists a locally $\alpha$-H\"older homeomorphism $f:\R^3 \to \mathbb{H}^1$  is   only known to lie between $1/2$ and $2/3$ \cite[Page 3]{SAD:P}. 

In this paper, we will consider a related problem. 
We first define a class of maps related to the class $C^{0,\alpha}(X;Y)$ of  $\alpha$-H\"older maps $f:X\to Y$.

\begin{definition}
Fix metric spaces $(X,d_X), (Y,d_Y)$ and $\alpha>0$. We say a map $f:X\to Y$ is of class $C^{0,\alpha+}(X;Y)$ if there exists 
a homeomorphism  $\beta:[0,\infty)\to [0,\infty)$  such that
\begin{equation}\label{calpha-ineq}
d_Y(f(a),f(b)) \le d_X(a,b)^\alpha  \beta(d_X(a,b)) \quad \text{for all } a,b\in X.
\end{equation}
We will sometimes simply write $C^{0,\alpha+}$ if the domain and target are clear. 
\end{definition}


\begin{remark}\label{alphapluslimit}
Suppose $X,$ $Y$ are metric spaces with $X$ bounded.
It is easy to check that
\[
C^{0,\eta}(X;Y) \subseteq C^{0,\alpha+}(X;Y) \subseteq C^{0,\alpha}(X;Y).
\]
whenever $0<\alpha<\eta$. 
Thus, $C^{0,\alpha+}(X;Y)$ can thought of as  a right limit of H\"older spaces.
\end{remark}

For certain models of model filiform groups and Carnot groups of small step, we will prove that there do not exist $(\alpha+)$-H\"{o}lder equivalences for $\alpha\ge 1/2$. 
Before stating our paper's two main results, we make the following definition.

\begin{definition}
A Carnot group $(\R^n,\cdot)$ is said to be \textbf{purely $k$-unrectifiable} if for every $A\subseteq \R^k$ and Lipschitz map $f:A \to (\R^n,\cdot)$, we have \[\mathcal{H}_{cc}^k(f(A))=0.\] Here,  we endow $(\R^n,\cdot)$ with the Carnot-Carath\'eodory metric to be described in subsection \ref{path-subsection}.
\end{definition}

Ambrosio and Kirchheim proved that  $\mathbb{H}^1$ is purely $k$-unrectifiable for $k=2,3,4$ \cite[Theorem 7.2]{RS:AK}.
More generally, Magnani proved that a Carnot group is purely $k$-unrectifiable if and only if its horizontal layer does not contain a Lie subalgebra of dimension $k$ \cite[Theorem 1.1]{UAR:M}.
In particular, $\mathbb{H}^n$ is purely $k$-unrectifiable for all $k>n$. 
In 2014, Balogh, Haj{\l}asz, and Wildrick provided a different proof of this last result by using approximate derivatives and a weak contact condition  \cite[Theorem 1.1]{WCE:BHW}. 
In the process, they prove that a Lipschitz mapping of an open subset of $\R^k$, $k>n$, into $\mathbb{H}^n$  has an approximate derivative that is horizontal almost everywhere.

Motivated by Gromov's H\"older equivalence problem, Balogh, Haj{\l}asz, and Wildrick  go on to prove that one cannot embed $\R^k$, $k>n$, into $\mathbb{H}^n$ via a sufficiently regular $(\alpha+)$-H\"older mapping.
More specifically, they prove that if $k>n$ and $\Omega \subseteq \R^k$ is open, then there is no injective mapping of class $C^{0,\frac{1}{2}+}(\Omega, \mathbb{H}^n)$ that is locally Lipschitz as a mapping into $\R^{2n+1}$ \cite[Theorem 1.11]{WCE:BHW}. 
The main key to their proof  is showing that if such a map existed, then it would have to be horizontal almost everywhere.
Notice that Remark \ref{alphapluslimit} combined with the identity map $\text{id}:\R^3\to \mathbb{H}^1$ being locally $\frac{1}{2}$-H\"older suggest that this result is sharp except for the extra local Lipschitz assumption.

In this paper, we will extend the  result in the previous paragraph to more general Carnot groups, specifically model filiform groups and Carnot groups of step at most three. 
The model filiform groups can be realized as the class of jet spaces $J^k(\R)$. In these groups, there are  few nontrivial bracket relations relative to the step.
For Carnot groups of small step, the Baker-Campbell-Hausdorff has a simple form; this allows one to  describe the  structure (e.g., left-invariant vector fields and contact forms) of the Carnot group in coordinates and perform  computations.
The Lie algebraic properties of these two classes of Carnot group make them ideal settings to generalize  the result from the previous paragraph.
The proofs for these Carnot groups will again boil down to showing the almost everywhere horizontality of certain $C^{0,\frac{1}{2}+}$ mappings  into these groups.

The standard basis $\{e^{(k)},e_k,\ldots , e_0\}$ of $Lie(J^k(\R))$ is such that $[e_j,e^{(k)}] = e_{j-1}$, $j\ge 1$, are the only nontrivial bracket relations. 
We will equip $J^k(\R)$ with coordinates of the first and second kind with respect to this basis.
For example, $J^1(\R)$ is isomorphic to $\mathbb{H}^1$. 
This will be discussed further in subsection \ref{jkr-carnot-subs}. 
It is implied that $J^k(\R)$ is equipped with either one of the two systems of coordinates in the following result, the first of our two main theorems.



\begin{theorem}\label{mf-main-theorem}
Fix $\alpha\ge \frac{1}{2} $ and positive integers $n, k$ with $n>1$. Let  $\Omega$ be an open subset of $\R^n$.
Then there is no injective mapping in the class
$C^{0,\alpha+}(\Omega; J^k(\R))$ that is also locally Lipschitz when considered as a map into $\R^{k+2}$.
\end{theorem}

We will   prove this result in the case $\alpha=\frac{1}{2}$, and
the cases for $\alpha>\frac{1}{2}$ will follow from the fact
\[
C^{0,\alpha+}(\Omega;J^k(\R)) \subset C^{0,\frac{1}{2}+}(\Omega;J^k(\R)).
\]

The identity map $\R^{k+2}\to J^k(\R)$ is locally $\frac{1}{k+1}$-H\"older. 
From the Heisenberg case, one may expect for it to be unknown whether there exist locally $\alpha$-H\"older, injective maps    $f:\R^{n}\to J^k(\R)$ for $\alpha>\frac{1}{k+1}$. 
However,  we will give an example of a locally $\frac{1}{2}$-H\"older, injective map  $f: \R^2\to J^k(\R)$   that is locally Lipschitz as a map into $\R^{k+2}$ (Example \ref{sharpness-example}). 
Comparing with Remark \ref{alphapluslimit}, this suggests that our result is  sharp, at least in the case $n=2$. 

We will first prove Theorem \ref{mf-main-theorem} for when $J^k(\R)$ is equipped with coordinates of the second kind.
We will then prove at  end of the subsection \ref{spi-subs} that this implies the theorem holds for first kind coordinates  as well. 
We will use Warhurst's model for jet spaces equipped with coordinates of the second kind (see \cite[Section 3]{JSA:W}). 
Rigot, Wenger, and Young have used this model to investigate   extendability of Lipschitz maps into jet spaces \cite{LE:WY, LNE:RW}. 

For the next result, we can choose coordinates with respect to any basis compatible with the stratification of $\mathfrak{g}$, but the metric on $(\R^n,\cdot)$ will be induced by this choice. 

\begin{theorem}\label{step-2-3-main-theorem}
Fix $\alpha\ge \frac{1}{2}$ and an open subset $\Omega\subseteq \R^k$.
Suppose $(\R^n,\cdot)$ is a Carnot group of step at most three that is purely $k$-unrectifiable.
Then there is no injective mapping in the class $C^{0,\alpha+}(\Omega; (\R^n,\cdot)) $ that 
is also locally Lipschitz when considered as a map into $\R^n$.
\end{theorem}

As for $J^k(\R)$, we will only explicitly prove this for $\alpha = \frac{1}{2}$.  
These two theorems will be proven in a similar fashion, implied by the following result:
\begin{proposition}\label{weakly-contact-prop}
Fix  an open subset $\Omega \subseteq \R^k$. 
Let $(\R^n,\cdot)$ be a Carnot group that is purely $k$-unrectifiable. 
Then there is no injective mapping $f:\Omega\to (\R^n,\cdot)$ that is weakly contact and
locally Lipschitz when considered as a map into $\R^n$.
\end{proposition}

Thus, to prove Theorems \ref{mf-main-theorem} and \ref{step-2-3-main-theorem}, it suffices to show  that if a map in  $C^{0,\frac{1}{2}+}(\Omega, (\R^n,\cdot))  $ is locally Lipschitz as a map into $\R^n$, then it is weakly contact.
We prove this for the class of model filiform jet spaces, $J^k(\R)$, in Proposition \ref{mf-diff-prop}, for step 2 Carnot groups in Lemma \ref{step-2-hor-lemma}, and for step 3 Carnot groups in Lemma \ref{step-3-hor-lemma}. 
We will discuss weakly contact maps further in subsection \ref{weak-cont-subs}.

Proposition \ref{mf-diff-prop} follows from considering the group structure on $J^k(\R)$, specifically Lemma \ref{jet-ind-lemma}.
The proofs of  Lemmas \ref{step-2-hor-lemma} and \ref{step-3-hor-lemma} are a bit technical and requires one to carefully work with  group  structures, bounding terms via the Ball-Box Theorem (Theorem \ref{bb-theorem}) and the modulating homeomorphism.
It is expected that Theorem \ref{mf-main-theorem} and \ref{step-2-3-main-theorem} should generalize to all Carnot groups if one attains a better understanding of the group structure arising from coordinates of the first kind. 
We will discuss this more  at the end of this paper.

\section{Background on Carnot groups}\label{back-carnot-groups}


In this section, we will review the basics of Carnot groups, discussing two systems of coordinates, the Carnot-Carath\'eodory metric, and weakly contact maps.

For some $r$, the Lie algebra $\mathfrak{g}$ of a Carnot group $G$ admits an $r$-step stratification: 
\[
\mathfrak{g}= \mathfrak{g}_1\oplus\mathfrak{g}_2\oplus \cdots \oplus \mathfrak{g}_r,
\]
where $\mathfrak{g}_1\subseteq \mathfrak{g}$ is a subspace, $\mathfrak{g}_{j+1} = [\mathfrak{g}_1,\mathfrak{g}_j]$ for all $j=1,\ldots , r-1$, and $[\mathfrak{g}, \mathfrak{g}_r] =0$. 
We write $[\mathfrak{g}_1,\mathfrak{g}_j]$ to denote the subspace generated by commutators of elements of $\mathfrak{g}_1$ with  elements of $\mathfrak{g}_j$, and similarly with $[\mathfrak{g},\mathfrak{g}_r]$. 
The subspaces $\mathfrak{g}_j$ are commonly referred to as the \textbf{layers} of $\mathfrak{g}$, with $\mathfrak{g}_1$ referred to as the \textbf{horizontal layer}.
We define the step of $G$ to be $r$, and this is well-defined \cite[Proposition 2.2.8]{SLG:BLU}.
Throughout this paper, we will implicitly fix a stratification for each Carnot group. In other words, we will view the stratification of $\mathfrak{g}$ as data of a Carnot group $G$. 

After combining bases of the subspaces $\mathfrak{g}_j$ to obtain a basis of $\mathfrak{g}$, we can define an inner product  $g= \langle \cdot , \cdot\rangle$ on $\mathfrak{g}$  by declaring the combined basis to be orthonormal.
Thus, we  say that a basis $\mathcal{B}= \{X^1,\ldots , X^n\}$ of $\mathfrak{g}$ is \textbf{compatible with the stratification of} $\mathfrak{g}$ if 
\[
\{X^{h_{j-1}+1}, \ldots , X^{h_j}\}
\]
is a basis of $\mathfrak{g}_j$ for each $j$, where $h_j = \sum_{i=1}^{j} \text{dim}(\mathfrak{g}_i)$. 
As we discuss coordinates of the first and second kind, it will be implied that coordinates are being taken with respect to a basis compatible with the stratification of $\mathfrak{g}$. 
While choosing different bases may technically result in different group structures, we will see that the resulting Carnot groups are all  isomorphic to $G$.


\subsection{Coordinates of the first kind}\label{first-coord-section}

For Carnot groups, the exponential map $\text{exp}:\mathfrak{g}\to G$ is a diffeomorphism \cite[Page 13]{CG:RON}.
Hence we can define $\star:\mathfrak{g}\times \mathfrak{g} \to \mathfrak{g}$ by
\[
X\star Y = \exp^{-1} (\exp(X)\exp(Y)).
\]
The Baker-Campbell-Hausdorff formula gives us an explicit formula for $X\star Y$:
\[
X\star Y = \sum_{n>0} \frac{(-1)^{n+1}}{n}  \sum_{0<p_i+q_i} \frac{1}{C_{p,q}}  (\ad X)^{p_1} (\ad Y)^{q_1} \cdots (\ad Y)^{q_{n-1}}W(p_n,q_n),
\]
where 
\[
C_{p,q} = p_1! q_1! \cdots p_n!q_n! \sum_{i=1}^n (p_i+q_i)
\]
and
\[
W({p_n},{q_n} ) = \left\{ \begin{array}{ll}
(\ad X)^{p_n} (\ad Y )^{q_n-1}Y, &\text{if } q_n\ge 1,\\
(\ad X)^{p_n-1} X , &\text{if } q_n=0.
\end{array}\right.
\]
The expansion of $X\star Y$ up to order $3$ is given by
\[
X+Y + \frac{1}{2} [X,Y] + \frac{1}{12} ([X,[X,Y]] + [Y,[Y,X]]). 
\]

Set $n$ equal to the topological dimension of $G$, and let $\mathcal{B}\subset \mathfrak{g}$ be a basis  compatible with the stratification of $\mathfrak{g}$. 
We can identify  $\mathfrak{g}$ with $\R^n$ via coordinates of $\mathcal{B}$, and then $\star$ on $\mathfrak{g}$ translates into an operation on $\R^n$.
With a slight abuse of notation, we will also denote this operation on $\R^n$ by $\star$. 
Then $(\R^n,\star)$ is a Carnot group  isomorphic to $G$ via $\exp$  \cite[Proposition 2.2.22]{SLG:BLU}. 
We say that $(\R^n,\star)$ is a \textbf{normal model of the first kind} of $G$ and that $(\R^n,\star)$ is $G$ equipped with \textbf{coordinates of the first kind with respect to $\mathcal{B}$}.
Observe that if $G$ is of step $r$, each coordinate of $X\star Y$ is a polynomial of homogeneous degree at most $r$ in the coordinates of $X$ and $Y$. 

\subsection{Path metric on Carnot groups}\label{path-subsection}

Let $(\R^n,\cdot)$ be a Carnot group, and set $m_j = \text{dim}(\mathfrak{g}_j)$ for each $j$.
Fix a basis  $\mathcal{B}_1=\{X^1,\ldots , X^{m_1}\}$ for  the horizontal layer $\mathfrak{g}_1$.
The \textbf{horizontal bundle}  of $(\R^n,\cdot)$ is defined fiberwise by
\[
H_p (\R^n,\cdot) := \text{span}\{X^1_p ,\ldots , X^{m_1}_p\}.
\]
Note the horizontal bundle is left-invariant: 
\[
H_p(\R^n,\cdot) = dL_p H_0(\R^n,\cdot).
\]
Declaring $(\mathcal{B}_1)_p$ to be orthonormal, we obtain an inner product on each fiber $H_p(\R^n,\cdot)$. 

Recall that we  just write  $\R^n$  to denote Euclidean space equipped with the standard Euclidean metric.

\begin{definition}
 We say a path $\gamma:[a,b]\to (\R^n,\cdot) $ is \textbf{horizontal} if it is absolutely continuous as a map into $\R^n$  and
\[
\gamma'(t) \in H_{\gamma(t)}(\R^n,\cdot) \quad\text{for a.e. } t\in [a,b].
\]
We define the \textbf{length} of a horizontal path to be 
\[
l_H(\gamma) := \int_a^b |\gamma'(t)|\ dt.
\]
Here, $|\gamma'(t)| := \sqrt{\langle \gamma'(t),\gamma'(t)\rangle_{\gamma(t)}}$ whenever  
$\gamma$ is differentiable at $t$ with 
$\gamma'(t) \in H_{\gamma(t)}(\R^n,\cdot)$. 
\end{definition}
Note the length of a horizontal path is finite (see 3.35, \cite{RA:F}).

A theorem by Chow \cite{USV:C} states that $(\R^n,\cdot)$ is horizontally path-connected. This enables us to define the \textbf{Carnot-Carath\'eodory metric} on $(\R^n,\cdot)$:
\[
d_{cc}(x,y) := \inf_{\gamma:[a,b]\to (\R^n,\cdot)} \{l_H(\gamma) : \gamma \text{ is horizontal} , \  \gamma(a)= x, \ \gamma(b) = y\}.
\]
Another common name for this metric is \emph{cc-metric}. 
It is well-known that the Carnot-Carath\'eodory metric defines a  geodesic metric on $(\R^n,\cdot)$, i.e., for every $x$, $y  \in (\R^n,\cdot)$, there exists a horizontal path $\gamma$ connecting $x$ to $y$ with $d_{cc}(x,y) = l_H(\gamma)$ \cite[Theorem 5.15.5]{SLG:BLU}.

Suppose $(\R^n,\cdot)$ is step $r$. 
From the previous two sections,
a point $x\in (\R^n,\cdot)$ is of the form $(\vec{x_1},\vec{x_2}\ldots , \vec{x_r})$, where  each $\vec{x_j}$ lies in $\R^{m_j}$ and corresponds to the coefficients of the elements of $\mathfrak{g}_j$. 
For each $\eps>0$, we can define a \textbf{dilation} $\delta_\epsilon:(\R^n,\cdot)\to (\R^n,\cdot) $ by 
\[
\delta_\epsilon(\vec{x_1},\vec{x_2}, \ldots , \vec{x_r}) := (\eps\vec{x_1}, \eps^2\vec{x_2} , \ldots , \eps^r \vec{x_r}).
\]
The Carnot-Carath\'eodory metric is left-invariant and one-homogeneous with respect to these dilations:

For all $\epsilon>0$ and $x,y,z\in (\R^n,\cdot)$, 
\begin{itemize}
 \item $d_{cc}(z\cdot x, z\cdot y )= d_{cc}(x,y)$
\item $d_{cc}(\delta_\epsilon(x) , \delta_\epsilon(y)) =\epsilon\cdot d_{cc}(x,y).$
\end{itemize}

One may wonder how the Carnot-Carath\'eodory metric on $(\R^n,\cdot)$ relates to the standard Euclidean metric on $\R^n$.
From Proposition \ref{holder}, $(\R^n,\cdot)$ and $\R^n$ have the same topologies. 
Furthermore, Proposition  \ref{holder} (combined with  left-invariance and homogeneity) implies  the following version of the Ball-Box Theorem:
\begin{theorem}{(Ball-Box Theorem)}\label{bb-theorem}
Suppose $(\R^n,\cdot)$ is a step $r$ Carnot group. 
For $\eps>0$ and $p\in (\R^n,\cdot)$, define 
\[
Box (\eps) :=\prod_{j=1}^r [-\eps^j,\eps^j]^{m_j} 
\]
and
\[
B_{cc}(p,\eps) := \{q\in (\R^n,\cdot): d_{cc}(p,q) \le \eps\}.
\]
There exists $C>0$ such that for all $\eps>0$ and $p\in (\R^n,\cdot)$,
\[
B_{cc} (p,\eps/C) \subseteq p\cdot Box(\eps) \subseteq B_{cc}(p,C\eps).
\]
\end{theorem}

We obtain an important corollary which allows us to estimate the cc-metric:
\begin{corollary}\label{carnot-bb-cor}
Suppose $(\R^n,\cdot)$ is a step $r$ Carnot group.
There exists $C>0$ such that for all $p=(a_1^1,\ldots , a_{m_1}^1,a^2_1,\ldots , a^r_{m_r})\in (\R^n,\cdot)$,
\[
\frac{1}{C} \cdot d_{cc}(0,p)\le \max\{|a^j_k|^{1/j}: 1\le j\le r, \ 1\le k \le m_j\}\le C d_{cc}(0,p).
\]
\end{corollary}

\subsection{Weakly contact Lipschitz mappings}\label{weak-cont-subs}
Fix  an open set $\Omega\subseteq \R^k$ and a Carnot group $(\R^n,\cdot)$.
If $f:\Omega \to (\R^n,\cdot)$ is Lipschitz,  $f$ is locally Lipschitz as a map into $\R^n$ by Proposition \ref{holder}.
By Rademacher's Theorem, then $f$ is differentiable almost everywhere in $\Omega$.
We say a  locally Lipschitz map $f:\Omega\to \R^n$ is \textbf{weakly contact} if
\[
\text{im } df_x \subset H_{f(x)} (\R^n,\cdot)\quad  \text{for } \mathcal{H}^k-\text{almost
 every } x\in \Omega.
\]
Here, we write $df_x$ to denote the \textbf{differential} or \textbf{total derivative} of $f$ at $x$. 
Observe that by Theorem 9.18 of \cite{PMA:R}, if $f$ is differentiable at $x\in \Omega$, then  
\[
\text{im } df_x \subset H_{f(x)} (\R^n,\cdot) \qquad \Longleftrightarrow \qquad \partial_i f(x)\in H_{f(x)}(\R^n,\cdot)  \ \text{for all }  i = 1,\ldots , k.
\]


Balogh, Haj{\l}asz, and Wildrick proved for the $n^{th}$ Heisenberg group $\mathbb{H}^n$ that if a Lipschitz map $f:[0,1]^k \to \R^{2n+1}$ is weakly contact, then it is actually Lipschitz as a map into $\mathbb{H}^n$ \cite[Proposition 8.2]{WCE:BHW}.
Their proof easily converts into a  statement for all Carnot groups.
To keep this paper as self-contained as possible, we will repeat the argument here.

\begin{proposition}\label{carnot-lip-weak-prop}
Let $k$ be a positive integer. If $f:[0,1]^k\to \R^n$ is Lipschitz and weakly contact,
then $f:[0,1]^k \to (\R^n,\cdot)$ is Lipschitz.
\end{proposition}

\begin{proof}
Fix a weakly contact map $f:[0,1]^k \to \R^n$ that is $L$-Lipschitz.
Fubini's Theorem implies the restriction of $f$ to almost every line segment parallel to a coordinate axis is horizontal. 
On bounded sets, the lengths with respect to the sub-Riemannian metrics and to the Euclidean metrics are equivalent for horizontal vectors.
As $f[0,1]^k$ is bounded and the Euclidean speed of $f$ is bounded by $L$ on line segments, it follows that the restriction of $f$ on almost every line segment parallel to a coordinate axis is $CL$-Lipschitz as a map into $(\R^n,\cdot)$. 
Hence  the restriction of $f$ on \emph{each} line segment parallel to a coordinate axis is $CL$-Lipschitz as a map into $(\R^n,\cdot)$, and the result follows.
\end{proof}

This enables us to prove  Proposition \ref{weakly-contact-prop}, a result fundamental to our paper. 
The proof of Theorem 1.11 in \cite{WCE:BHW} for the Heisenberg group translates into a result for all Carnot groups. 

\begin{proof}[Proof of Proposition \ref{weakly-contact-prop}]
Assume that there is an injective map $f:\Omega\to (\R^n,\cdot)$ that is locally Lipschitz as
a map into $\R^n$.
Restricting $f$, we may assume $\Omega$ is a closed cube and $f$ is Lipschitz as a map into $\R^n$.
If $f$ is weakly contact, $f:\Omega\to (\R^n,\cdot)$ is Lipschitz,
which implies $\mathcal{H}^k_{(\R^n,\cdot)} (f(\Omega))=0$.
As the identity map from $(\R^n,\cdot)$ to $\R^n$ is locally Lipschitz (by Proposition \ref{holder}), 
$\mathcal{H}^k_{\R^n} (f(\Omega)) =0$.
It follows from Theorem 8.15 of \cite{LOA:H}  that the topological dimension of $f(\Omega)$ is at most $k-1$.
Since $f|_{\Omega}$  is a homeomorphism, $f(\Omega)$ is  of the same topological dimension as $\Omega$, which is a contradiction.
\end{proof}

The main theorems of this paper thus reduce to showing  locally Lipschitz maps
$f$  into $\R^n$ that are of class $C^{0,\frac{1}{2}+}(\Omega,(\R^n,\cdot))$,  are weakly contact. 
Balogh, Haj{\l}asz, and Wildrick proved this for the Heisenberg group \cite[Proposition 8.1]{WCE:BHW}.
In this paper, we will prove it for models of  jet spaces and models of Carnot groups of step at most three.

\subsection{Strata-preserving isomorphisms}\label{spi-subs}
Suppose $G$ is a Carnot group with stratification
\[
\mathfrak{g}= \mathfrak{g}_1\oplus \cdots \oplus \mathfrak{g}_r.
\]
We define the family of dilations $\{\mathfrak{d}_\epsilon\}_{\epsilon>0}$  to be the collection of isomorphism of  $\mathfrak{g}$ induced by $\mathfrak{d}_\epsilon(X_j) = \epsilon^jX_j$, $X_j \in\mathfrak{g}_j$.
Each $\mathfrak{d}_\epsilon$ is a Lie group automorphism of $(\mathfrak{g},\star)$ \cite[Remark 1.3.32]{SLG:BLU}, i.e., 
\begin{equation}
\mathfrak{d}_\epsilon (X\star Y) = (\mathfrak{d}_\epsilon(X)) \star (\mathfrak{d}_\epsilon (Y)) \quad \text{for all }X,Y\in \mathfrak{g}. \label{auto-dil}
\end{equation}
 These dilations on $\mathfrak{g}$ are also commonly notated as $\delta_\epsilon$, but we will not do so here to avoid confusion with the dilations on $G$.

As the exponential map $\exp_G:\mathfrak{g}\to G$ is a diffeomorphism, this induces a family of dilations $\delta_\epsilon$ on $G$:
\begin{equation}
\delta_\epsilon := \exp_G \circ \mathfrak{d}_\epsilon\circ \exp_G^{-1}. \label{group-dil}
\end{equation}
This aligns with our earlier definition of dilations in subsection \ref{path-subsection}.

Suppose $H$ is a Carnot group isomorphic to $G$,  with stratification
\[
\mathfrak{h} = \mathfrak{h}_1\oplus \cdots \oplus \mathfrak{h}_r.
\]
A Lie group isomorphism $\varphi:G\to H$ induces a Lie algebra isomorphism $\varphi_*:\mathfrak{g}\to \mathfrak{h}$  that satisfies the following identity:
\begin{equation}
\exp_H\circ \varphi_*= \varphi\circ \exp_G.\label{nat-exp}
\end{equation} 
We say that a Lie group isomorphism $\varphi:G\to H$ \textbf{commutes with dilation} if 
\[
\varphi (\delta_\epsilon^G g) = \delta_\epsilon^H \varphi(g) \quad \text{for all } g\in G,\ \epsilon>0,
\]
where $\delta_\epsilon^G, \delta_\epsilon^H$ denote the dilations on $G$, $H$, respectively. 
If we  say that a Lie algebra isomorphism $f:\mathfrak{g}\to \mathfrak{h}$ commutes with dilation if 
\[
f(\mathfrak{d}_\epsilon^G X) = \mathfrak{d}_\epsilon^H f(X) \quad\text{for all } X\in \mathfrak{g}, \ \epsilon>0,
\]
it is easy to check using (\ref{group-dil}) and (\ref{nat-exp}) that an isomorphism $\varphi:G\to H$ commutes with dilations if and only if $\varphi_*:\mathfrak{g}\to \mathfrak{h}$ commutes with dilations.

\begin{example}
Let $G$ be a Carnot group.
Suppose $\mathcal{B}\subset \mathfrak{g}$ is a basis compatible with the stratification of $\mathfrak{g}$. 
Let $(\R^n,\odot)$ and $(\R^n,\star)$ be $G$ equipped with coordinates of the second and first kind, respectively, with respect to $\mathcal{B}$. 
Then $(\R^n,\odot) $ is isomorphic to $(\R^n,\star)$ via $\exp^{-1}\circ \Phi$ and coordinates.
Moreover, this isomorphism commutes with dilations.  
\end{example}

We say that an isomorphism $\varphi :G\to H$ is \textbf{strata-preserving} if 
\[
\varphi_*(\mathfrak{g}_j) = \mathfrak{h}_j \quad\text{for all } j=1,\ldots , r.
\]
Note that $\varphi$ is strata-preserving if and only if $\varphi^{-1}$ is strata-preserving.

The next result follows from  the use of  dilations:

\begin{lemma}\label{dil-strat-prop}
Let $G$, $H$ be isomorphic Carnot groups. 
An isomorphism $\varphi:G\to H$ commutes with dilations if and only if $\varphi$ is strata-preserving.
\end{lemma}

In fact, if we say that an isomorphism $\varphi:G\to H$ is \emph{contact} if $\varphi_*(\mathfrak{g}_1) = \mathfrak{h}_1$, it's easy to check from the stratifications of $\mathfrak{g}$ and $\mathfrak{h}$ that $\varphi$ is a contact map if and only if it is strata-preserving.

We will show weakly contact mappings are invariant under isomorphisms that commute with dilations.
We first prove that such isomorphisms are  biLipschitz.
\begin{proposition}\label{iso-loc-bil-prop}
Let $\varphi:(\R^n,\cdot)\to (\R^n,\ast)$ be an isomorphism between Carnot groups, that commutes with dilations. Then $\varphi$ is biLipschitz, i.e., there exists a constant $C$ such that 
\[
\frac{1}{C}d_{cc}^{(\R^n,\cdot)}(g,h) \le d_{cc}^{(\R^n,\ast)} (\varphi(g),\varphi(h)) \le C d_{cc}^{(\R^n,\cdot)}(g,h) \quad\text{for all } g,h\in (\R^n,\cdot). 
\] 

\end{proposition}

\begin{proof}
As $\varphi$ commutes with dilations and the cc-metrics on $(\R^n,\cdot)$ and $(\R^n,\ast)$ are one-homogeneous, it suffices to show $\varphi$ is biLipschitz when restricted to $B_{cc}(e,1)$. 

Let $\{X^1,\ldots , X^{m_1}\}$, $\{Y^1,\ldots , Y^{m_1}\}$  be left-invariant  frames for $H(\R^n,\cdot)$, $H(\R^n,\ast)$, respectively.
 For each $g\in (\R^n,\cdot)$,  define the linear isomorphism $S_g: H_{\varphi(g)} (\R^n,\ast) \to H_{\varphi(g)} (\R^n,\ast)$ induced  by
$
(\varphi_*X^j)_{\varphi(g)} \mapsto Y^j_{\varphi(g)}.
$  
The function $g\mapsto ||S_g||$ is continuous, and hence, is bounded on $B_{cc}(e,2)$, say by $C$. 
This implies for all $g\in B_{cc}(e,2)$ and $v \in H_g(\R^n,\cdot)$,  we have
$
|d\varphi_g (v)|_{\varphi(g)}\le C |v|_g.
$
It then follows from Lemma \ref{dil-strat-prop} that 
\[
d_{cc}^{(\R^n,\ast)}(\varphi(g),\varphi(h)) \le Cd_{cc}^{(\R^n,\cdot)}(g,h) 
\]
for all $g,h\in B_{cc}(e,1)$. 
Applying this argument to $\varphi^{-1}$, the lemma follows.
\end{proof}

It follows from the chain rule that weak contactness is preserved by strata-preserving isomorphisms.

\begin{corollary}\label{cor-iso-weak-cont}
Fix $\Omega\subseteq \R^k$ an open subset. 
Let $\varphi:(\R^n,\cdot)\to (\R^n,\ast)$ be an isomorphism between Carnot groups, that commutes with dilations. If $f:\Omega\to (\R^n,\cdot)$ is  locally Lipschitz  and weakly contact, then $\varphi\circ f :\Omega\to (\R^n,\ast)$ is also locally Lipschitz and weakly contact. 
\end{corollary}

\subsection{Coordinates of the second kind}\label{second-coord-section}

Now that we have defined dilations on $G$ and $\mathfrak{g}$, we can introduce coordinates of the second kind, another model for Carnot groups. The Carnot group that we obtain via this construction will be isomorphic to the coordinates of the first kind model we described in subsection \ref{first-coord-section}. 
We will first state a result that will allow us to define our other model.

\begin{theorem}\label{thm-second-coord}(\cite[Theorem 2.10.1]{LG:V})
Let $G$ be a Lie group with Lie algebra $\mathfrak{g}$. 
Suppose $\mathfrak{g}$ is the direct sum of linear subspaces $\mathfrak{h}_1 , \ldots, \mathfrak{h}_s$. 
Then there are open neighborhoods $B_i$ of $0$ in $\mathfrak{h}_i $ ($1\le i\le s$) and $U$ of $1$ in $G$, such that the map 
\[
\Psi:(Z_1,\ldots   , Z_s) \mapsto \exp Z_1\cdots \exp Z_s
\]
is an analytic diffeomorphism of $B_1\times \cdots \times B_s$ onto $U$. 
\end{theorem}

Fix a basis $\mathcal{B}= \{X^1,\ldots , X^n\}$  of $\mathfrak{g}$ compatible with the stratification, and define $\Phi:\mathfrak{g}\to G$  by 
\[
\Phi (a_1X^1 + \cdots   + a_nX^n ) = \exp (a_1X^1)\cdots \exp (a_nX^n).
\]
By Theorem \ref{thm-second-coord}, the restriction $\Phi|_V:V\to U$ is a diffeomorphism for some open neighborhoods  $V\subset \mathfrak{g}$ of $0$  and $U\subset G$ of $e$. 
After noticing
$
\Phi(a_1X^1+\cdots + a_nX^n) = \exp(a_1X^1\star \cdots \star a_nX^n),
$
it follows from (\ref{auto-dil})  and (\ref{group-dil})   that $\Phi$ is a global diffeomorphism.

 We can then define $\odot:\mathfrak{g}\times \mathfrak{g}\to \mathfrak{g}$ by
\[
X\odot Y  = \Phi^{-1} (\Phi(X)\Phi(Y)). 
\]
As in subsection \ref{first-coord-section}, we can identify $\mathfrak{g}$ with $\R^n$ and define a corresponding operation $\odot$ on $\R^n$, with a slight abuse of notation.
We say that $(\R^n, \odot)$ is a \textbf{normal model of the second kind} of $\mathfrak{g}$, and $(\R^n,\odot)$ is $G$ equipped with \textbf{coordinates of the second kind with respect to $\mathcal{B}$}.
Identifying $(\R^n,\star)$ with $\mathfrak{g}$ via the same basis, observe that  $\exp^{-1}\circ \Phi:( \R^n,\odot)\to (\R^n,\star)$ is a Lie group isomorphism.
 In particular, $(\R^n,\odot)$ is  isomorphic to $G$.
It then follows from Corollary \ref{cor-iso-weak-cont} that it suffices to prove each of  Theorems \ref{mf-main-theorem} and  \ref{step-2-3-main-theorem} for a single system  of coordinates. 




\section{Result for  $J^k(\R)$}\label{jkr-section}

\subsection{$J^k(\R)$ as Carnot groups}\label{jkr-carnot-subs}

We will only do our discussion in this section for jet spaces $J^k(\R)= J^k(\R,\R)$  (for $k\ge 1$)   to make things clearer.
Similar constructions can be used to define  more general jet spaces $J^k(\R^m,\R^n)$ (see \cite[Section 4]{JSA:W}).
The results in this paper concerning model filiform groups  translate into results for general jet spaces, and
I will note the more general results. 

Given $f,g\in C^k(\R)$, we say $f$ is equivalent to $g$ at $x\in \R$, and write $f\sim_x g$, if their $k^{th}$-order Taylor polynomials agree at $x$. 
Define
\[
J^k(\R) = \bigcup_{x\in \R} C^k(\R)/ \sim_x,
\]
and observe we have  global coordinates on $J^k(\R)$ by
\[
J^k(\R) \ni [f]_{\sim_x} \mapsto (x,u_k,\ldots  , u_0)\in \R^{k+2},
\]
where $u_j := f^{(k)}(x)$.

The horizontal bundle $HJ^k(\R)$ is defined pointwise by
\[
H_p J^k(\R) = \{ v\in T_p J^k(\R)| \ \omega_i (v) = 0, \ i = 0,\ldots , k-1\},
\]
where
\[
\omega_i:= du_i - u_{i+1} dx.
\]
In coordinates, $HJ^k(\R)$ is a $2$-dimensional tangent distribution on $J^k(\R)$ with global frame $\{X^{(k)}, \frac{d}{du_k}\}$, where 
\[
X^{(k)} = \frac{\partial}{\partial x} + u_k \frac{\partial}{\partial u_{k-1}} + \cdots + u_1\frac{\partial}{\partial u_0}.
\]
The nontrivial bracket relations are 
\[
\left[ \frac{\partial }{\partial u_j} , X^{(k)}\right] = \frac{\partial }{\partial u_{j-1}}, \quad j= 1, \ldots , k.
\]
 It follows that 
\[
Lie(J^k(\R)) = HJ^k(\R) \oplus \text{span}\left\{ \frac{\partial }{\partial u_{k-1}}\right\} \oplus \cdots \oplus
\text{span}\left\{ \frac{\partial }{\partial u_0}\right\}
\]
is a $(k+1)$-step stratified  Lie algebra.

One can use coordinates of the second kind to turn $J^k(\R)$ into a Carnot group with the following group operation:
\[
(x,u_k,\ldots ,u_0) \odot (y,v_k,\ldots , v_0) = (z,w_k, \ldots , w_0),
\]
where $z=x+y$, $w_k =u_k +v_k$, and 
\[
w_s = u_s+v_s + \sum_{j=s+1}^k u_j \frac{y^{j-s}}{(j-s)!}, \quad s=0,\ldots , k-1
\]
(see \cite[Example 4.3]{JSA:W}).
For $(x,u_k,\ldots ,u_0)\in J^k(\R)$, it is easy to show
\[
((x,u_k,\ldots , u_0)^{-1})_s = -\sum_{j=s}^k\frac{(-x)^{j-s}}{(j-s)!} u_j, \quad s=0,\ldots , k.
\]



\subsection{A horizontality result for  $J^k(\R)$}
In this section, we will prove  a horizontality condition for $J^k(\R)$, from which Theorem \ref{mf-main-theorem} will follow. We begin with a crucial lemma  concerning the group structure of $J^k(\R)$, similar to Corollary 1.3.18 of \cite{SLG:BLU}.

\begin{lemma}\label{jet-ind-lemma}
For $(x,u_k,\ldots , u_0) , $ $ (y,v_k,\ldots , v_0)\in J^k(\R)$,
\[
((x,u_k,\ldots , u_0) ^{-1}\odot(y,v_k,\ldots , v_0))_0=v_0 - u_0 - \sum_{j=1}^k \frac{u_j}{j!} (y-x)^j.
\]
\end{lemma}

\begin{proof}
Recall
\[
((x,u_k, \ldots , u_0)^{-1})_s = - \sum_{j=s}^k \frac{ (-x)^{j-s} }{(j-s)!} u_j, \quad s= 0,\ldots , k,
\]
and the last coordinate of $(x,u_k,\ldots , u_0) \odot (y,v_k,\ldots , v_0)$ is 
\[
((x,u_k,\ldots , u_0)\odot ( y,v_k,\ldots , v_0))_0=v_0+ \sum_{s=0}^k \frac{y^s}{s!} u_s.
\]
Thus, 
\begin{align*}
((x,u_k,\ldots , u_0)^{-1}\odot ( y,v_k,\ldots , v_0))_0 &= v_0 - \sum_{s=0}^k \sum_{j=s}^n \frac{y^s}{s!} \cdot 
\frac{(-x)^{j-s}}{(j-s)!} \cdot u_j\\
&= v_0 - \sum_{j=0}^k \sum_{s=0}^j \binom{j}{s} y^s(-x)^{j-s} \cdot \frac{u_j}{j!}\\
&= v_0 - \sum_{j=0}^k \frac{1}{j!}\cdot (y-x)^j u_j,
\end{align*}
where the last equality comes from the Binomial Theorem.
\end{proof}

\begin{remark}\label{gen-jet-space-rem}
The same reasoning using the Multinomial Theorem gives us the following generalization for all jet spaces:

Fix positive integers $k,m,n$. 
Let the notation for $J^k(\R^m,\R^n)$ be as in Warhurst (see \cite[Subsection 4.4]{JSA:W}), and 
equip $J^k(\R^m,\R^n)$ with the group operation arising from coordinates of the second kind (see  \cite[subsection 4.4]{JSA:W}).
Given  $(x,u^{(k)}), (y,v^{(k)}) \in J^k(\R^m,\R^n)$, 
\[
((x,u^{(k)})^{-1}\odot (y,v^{(k)}))^l_0= v_0^l-\sum_{I\in \tilde{I}(m)}\frac{u_I^l}{I!} (y-x)^I, \quad l=1,\ldots n.
\]
Here, for  $I =(i_1,\ldots , i_m)\in \tilde{I}(m)$ and $z=(z_1,\ldots , z_m)\in \R^m$, we define
$I!= i_1!\cdots i_m!$ and $z^I = z_1^{i_1}\cdots z_m^{i_m}$. 
\end{remark}

\begin{proposition}\label{mf-diff-prop}
Let $k,$ $ n$ be positive integers with  $\Omega\subseteq \R^n$ an open set. 
Suppose that $f= (f^x, f^{u_k}, \ldots , f^{u_0}):\Omega\to J^k(\R)$ is of class $C^{0,\frac{1}{2}+}.$
If the component $f^x $ is differentiable at a point $p_0\in \Omega$, then the components $f^{u_{k-1}},f^{u_{k-2}}, \ldots , f^{u_0}$
are also differentiable at $p_0$ with 
\[
df_{p_0}^{u_j} = f^{u_{j+1}}(p_0) df_{p_0}^x
\]
for all $j=0,\ldots , k-1$.
In particular, if $f^{u_k}$ is also differentiable at $p_0$, 
then the image of $df_{p_0}$ lies in the horizontal space $H_{f(p_0)}J^k(\R)$.
\end{proposition}

\begin{proof}
We prove this result by induction on $k\ge 1$. Below, $p$ is a point in $\Omega$. 

Let $f= (f^x,f^{u_1}, f^{u_0}):\Omega \to J^1(\R)$ be given of class $C^{0,\frac{1}{2}+}$.
Choose a map $\beta$ for $f$ satisfying (\ref{calpha-ineq}).
By Lemma \ref{jet-ind-lemma}, 
\[
(f(p_0)^{-1}f(p))_0= f^{u_0}(p)-f^{u_0}(p_0)- f^{u_1}(p_0)(f^{x}(p)-f^{x}(p_0)).
\]
Thus by Corollary \ref{carnot-bb-cor}, there exists $C>0$ such that
\begin{align*}
|f^{u_0}(p)-f^{u_0}(p_0)- f^{u_1}(p_0)(f^{x}(p)-f^{x}(p_0))|^{1/2} &\le Cd_{cc}(f(p),f(p_0))\\ & \le C \beta(|p-p_0|) \cdot |p-p_0|^{1/2}.
\end{align*}
We have
\begin{align*}
&|f^{u_0}(p) - f^{u_0}(p_0) - f^{u_1}(p_0)df_{p_0}^x(p-p_0)|\\
&\qquad \le C^2\beta^2(|p-p_0|) \cdot |p-p_0| + |f^{u_1}(p_0)(f^{x}(p)-f^{x}(p_0))-  f^{u_1}(p_0)df_{p_0}^x(p-p_0)|\\
&\qquad=o(|p-p_0|),
\end{align*}
where we used the differentiability of $f^x$ at $p_0$  for the last equality.

Suppose we have proven the result up to  $k$.
Let $f= (f^x,f^{u_{k+1}}, \ldots , f^{u_0}):\Omega\to J^{k+1}(\R)$ be given  of class $C^{0,\frac{1}{2}+}$
with $f^x$ differentiable at $p_0$.
Let $\tilde{\beta}$ be a map satisfying (\ref{calpha-ineq}) for $f$. 
Define the projection
 $\pi:J^{k+1}(\R)\to J^k(\R)$ by 
\[
 \pi(x,u_{k+1},\ldots , u_0) = (x,u_{k+1},\ldots , u_1).
\]
As $\pi$ maps horizontal curves to horizontal curves of the same length, it's not hard to see that $\pi$ is
a contraction. This implies $\pi\circ f= (f^x,f^{u_{k+1}}, \ldots , f^{u_1})$ is of class
$C^{0,\frac{1}{2}+}(\Omega,J^k(\R))$. By induction, $f^{u_{k}},\ldots , f^{u_1}$ are differentiable at $p_0$ 
with 
\[
df_{p_0}^{u_j} = f^{u_{j+1}}(p_0) df_{p_0}^x
\]
for all $j=1,\ldots ,k$. 

It remains to show $f^{u_0}$ is also differentiable at $p_0$ with 
\[
df_{p_0}^{u_0} = f^{u_1}(p_0)df_{p_0}^x.
\]

 Lemma \ref{jet-ind-lemma} and Corollary \ref{carnot-bb-cor} combine to imply
\[
\left|f^{u_0}(p) - f^{u_0}(p_0)- \sum_{j=1}^{k+1} \frac{f^{u_j}(p_0)}{j!} (f^x(p)-f^x(p_0))^j\right|^{1/(k+1)}\le C\tilde{\beta}(|p-p_0|)\cdot
|p-p_0|^{1/2}.
\]
Moreover, as $f^x$ is differentiable at $p_0$, 
\[
f^x(p)-f^x(p_0) = O(|p-p_0|),
\]
and hence
\[
|f^x(p)-f^x(p_0)|^j= o(|p-p_0|) \quad\text{for all } j\ge 2.
\]
It follows 
\begin{align*}
&|f^{u_0}(p)-f^{u_0}(p_0) - f^{u_1}(p_0) df^x_{p_0}(p-p_0)|\\
&\qquad \le C^{k+1}\tilde{\beta}^{k+1}(|p-p_0|)\cdot |p-p_0|^{\frac{k+1}{2}} + |f^{u_1}(p_0)(f^x(p)-f^x(p_0))-
f^{u_1}(p_0)df^x_{p_0}(x-x_0)| \\ &\qquad\qquad +
\sum_{j=2}^{k+1} \left|\frac{f^{u_j}(p_0)}{j!} (f^x(p)-f^x(p_0))\right|^j \\
&\qquad = o(|x-x_0|).
\end{align*}
This proves $f^{u_0}$ is differentiable at $p_0$ with 
\[
df^{u_0}_{p_0}=f^{u_1}(p_0)df_{p_0}^x,
\]
and the proposition follows.
\end{proof}

\begin{remark}
In the above proof, we needed $f$ to lie in $C^{0,\frac{1}{2}+}(\Omega,J^k(\R))$ in order to ensure $f^{u_{k-1}}$ was differentiable at the point with the desired form. To prove the differentiability of the components of $f$ corresponding to higher layers, one can assume lower regularity. In fact, 
the above proof shows the following:

Assume $\Omega\subseteq \R^n$ is open and $j\ge 2$. Suppose $f= (f^x,f^{u_k},\ldots , f^{u_0}):\Omega \to J^k(\R)$
is of class $C^{0,\frac{1}{j}+}$. If $f^x$ is differentiable at a point $p_0\in \Omega$, then $f^{u_{k+1-j}},
f^{u_{k-j}}, \ldots , f^{u_0}$ are also differentiable at $p_0$ with 
\[
df^{u_l}_{p_0} = f^{u_{l+1}}(p_0) df_{p_0}^x, \quad l = k+1-j,\ldots ,0.
\]

\end{remark}

\subsection{Proof of Theorem   \ref{mf-main-theorem}}\label{mf-main-subs}

Before we prove Theorem \ref{mf-main-theorem}, we will give an example of a locally $\frac{1}{2}$-H\"older map  $f:\R^2\to J^k(\R)$ that is Lipschitz as a map into $\R^{k+2}$. Comparing with Remark \ref{alphapluslimit}, this suggests that our result is sharp in the case $n=2$.

\begin{example}\label{sharpness-example}
Define $f:\R^2 \to J^k(\R)$ by 
\[
f(x,y) =  (0,x,y,0,\ldots , 0).
\]
Then $f$ is Lipschitz (in fact, is an isometry) as a map into $\R^{k+2}$. 

To show $f$ is locally $\frac{1}{2}$-H\"older,  first note in $J^k(\R)$,
\[
(0,-x_1,-y_1,0,\ldots , 0)\odot (0,x_2,y_2,0,\ldots ,0)=(0,x_2-x_1,y_2-y_1,0,\ldots , 0).
\]
By Corollary \ref{carnot-bb-cor}, there exists a constant $C$ such that
\[
 d_{cc}(f(x_1,y_1),f(x_2,y_2)) \le C\max\{|x_2-x_1|, |y_2-y_1|^{1/2}\}
\]
for all $(x_1,y_1),(x_2,y_2)\in \R^2$. 
By considering cases, one can   then show 
\[
d_{cc}(f(x_1,y_1),f(x_2,y_2))\le \sqrt{2M}C |(x_1,y_1)-(x_2,y_2)|^{1/2}
\]
whenever $(x_1,y_1),(x_2,y_2)\in [-M,M]^2$ with $M>1$.

\end{example}

\begin{proof}[Proof of Theorem \ref{mf-main-theorem}]
Fix positive integers $n,k$ with $n\ge 2$. 
Suppose  $f:\Omega\to J^k(\R)$ is of class $C^{0,\frac{1}{2}+}$ and
is locally Lipschitz as a map into $\R^{k+2}$.
By Rademacher's Theorem, each of the components of $f$ is differentiable almost everywhere, and in particular, $f^x$ is differentiable almost everywhere.
Proposition \ref{mf-diff-prop} then implies that $f$ is weakly contact. 
Since $J^k(\R)$ is purely $n$-unrectifiable \cite[Theorem 1.1]{UAR:M}, Theorem \ref{mf-main-theorem} in the case of second kind coordinates follows from Proposition \ref{carnot-lip-weak-prop}. 
The discussion at the end of subsection \ref{second-coord-section} then proves the result for coordinates of the first kind.
\end{proof}

\begin{remark}
Observe that  $J^k(\R^m,\R^n)$ is purely $j$-unrectifiable if $j>\binom{m+k-1}{k}$ \cite[Theorem 1.1]{UAR:M}.
Hence, from Remark \ref{gen-jet-space-rem}, one can use similar reasoning  to show the following generalization:

Fix a jet space $J^k(\R^m,\R^n)$ and equip it with the group structure from Subsection 4.4 of \cite{JSA:W}. Suppose  $j>\binom{m+k-1}{k}$ and $\Omega$  is an open subset of $\R^j$.
If $N $ is the topological dimension of $J^k(\R^m,\R^n)$, there is no injective mapping in the class
$C^{0,\frac{1}{2}+}(\Omega; J^k(\R^m,\R^n))$ that is also locally Lipschitz when considered as a map into $\R^{N}$.
\end{remark}

\begin{remark}
Theorem \ref{mf-main-theorem} has an easier proof if we assume $ n<\frac{1}{2}\left(1+ \frac{(k+1)(k+2)}{2}\right).$
Making this assumption,
suppose that $f:\Omega\to J^k(\R)$ is injective and of class $C^{0,\frac{1}{2}+}$.
Let $B(x,r)$ be an open  ball with $\overline{B(x,r)}\subseteq \Omega$.
Then the restriction $f|_{\overline{B(x,r)}}$ is injective and of class $C^{0,\frac{1}{2}+}(\overline{B(x,r))},
J^k(\R))$. Since $\overline{B(x,r)}$ is bounded, it follows that $f|_{\overline{B(x,r)}} $ is 
a $\frac{1}{2}$-H\"older homeomorphism.
In particular, $f(B(x,r)) $ is open in $J^k(\R)$, which implies
\[
\text{dim}_{\text{Hau}}f(B(x,r)) = \text{dim}_{\text{Hau}} J^k(\R) = 1+\frac{(k+1)(k+2)}{2}.
\]
But as $f$ is $\frac{1}{2}$-H\"older,
\[
\text{dim}_{\text{Hau}}f(B(x,r))  \le 2 \cdot\text{dim}_{\text{Hau}} B(x,r) = 2n,
\]
which is a contradiction.
\end{remark}

\section{Result for Carnot groups of  step at most three}

\subsection{Geometry of step two Carnot groups}\label{step-2-setup}

In this section, we will consider the geometry of Carnot groups of step two  and equip
these groups with coordinates of the first kind.

Fix a step two Carnot group $G$ with Lie algebra $\mathfrak{g}$. Writing 
$
\mathfrak{g}= \mathfrak{g}_1\oplus \mathfrak{g}_2,
$
let $d_1,\ldots , d_r$ be a basis for $\mathfrak{g}_1$ and $e_1,\ldots , e_s$ be a basis for $\mathfrak{g}_2$.
We can write 
\[
[d_i,d_j] = \sum_{k=1}^s \alpha_k^{ij} e_k
\]
for some structural constants $\alpha_k^{ij}$, with all other bracket relations  trivial.
By antisymmetry,
$
\alpha_{k}^{ij}  =- \alpha_k^{ji}
$ for all $i,$ $j,$ and $k$.
In fact, Bonfiglioli, Lanconelli, and Ugozzoni prove that there exists a Carnot group of step two with these bracket relations if and only if the skew-symmetric matrices $(\alpha_k^{ij})$, $k=1,\ldots , s$, are linearly independent \cite[Proposition 3.2.1]{SLG:BLU}. 

Using the procedure described in subsection \ref{first-coord-section}, we can identify $G$ with $\R^{r+s}$ equipped with the following multiplication via coordinates of the first kind:
\[
(A_1,\ldots  ,A_r,B_1,\ldots ,B_s) \star (a_1,\ldots , a_r,b_1,\ldots , b_s) = (\mathcal{A}_1,\ldots , \mathcal{A}_r, \mathcal{B}_1,\ldots , \mathcal{B}_s),
\]
where
\[
\mathcal{A}_i = A_i +a_i, \quad \mathcal{B}_k = B_k+b_k +\frac{1}{2} \sum_{1\le i <j \le r} \alpha_k^{ij}(A_ia_j-a_iA_j).
\]
We write $(A_1,\ldots  ,A_r,B_1,\ldots ,B_s) = (A_i, B_k)$ henceforth,
and we will use similar notation for step $3$ Carnot groups. 

By translating the canonical basis at the basis, we obtain  the left-invariant vector fields
\begin{align*}
&X^i := \frac{\partial}{\partial A_i} + \frac{1}{2} \sum_{k=1}^s \left( \sum_{j< i }\alpha_k^{ji} A_j - \sum_{j> i} \alpha_k^{ij}A_j\right) \frac{\partial}{\partial B_k}, \quad i=1,\ldots , r,\\
&Y^k := \frac{\partial}{\partial B_k}, \quad  k = 1,\ldots , s.
\end{align*}
We obtain the stratification
\[
Lie(\R^{r+s},\star)  = \langle X^i\rangle_{1\le i\le r}\oplus \langle Y^k\rangle_{1\le k \le s}
\]
\cite[Remark 1.4.8]{SLG:BLU}.
In fact, it is easy to check that  the linear map $\varphi: Lie(\R^{r+s},\star)\to \mathfrak{g}$ induced by $X^i\mapsto d_i$, $Y^k \mapsto e_k$ is a Lie algebra isomorphism. 

The contact forms, satisfying
\[
H(\R^{r+s},\star)= \bigcap_{k=1}^s \ker \omega^k,
\]
 are given by
\[
\omega^k := dB_k - \frac{1}{2}\sum_{i=1}^r \left( \sum_{j<i }\alpha_k^{ji} A_j - \sum_{j>i} \alpha_k^{ij}A_j\right) dA_i.
\]
In other words, if    $v\in T_{p}(\R^{r+s},\star)$,   then 
\[
v\in H_p(\R^{r+s},\star)\qquad \Longleftrightarrow \qquad \omega^k_p(v) = 0 \ \text{for all } k=1,\ldots ,s.
\]

\subsection{Geometry of step three Carnot groups}\label{step-3-setup}

Let $G$ be a step three Carnot group. 
Let $d_1,\ldots  , d_r$ be a basis for $\mathfrak{g}_1$, $e_1,\ldots , e_s$ a basis for $\mathfrak{g}_2$,
and $f_1,\ldots , f_t$ a basis for $\mathfrak{g}_3$. We  write 
\begin{align*}
&[d_i,d_j] = \sum_{k=1}^s \alpha_k^{ij} e_k\\
&[d_i , e_k] = \sum_{m=1}^t \beta_m^{ik} f_m
\end{align*}
with all other bracket relations trivial. 

As in the step two case, we can identify $G$ with $\R^{r+s+t}$ equipped
with the following operation via coordinates of the first kind:
\[
(A_i,B_k , C_m ) \star (a_i,b_k,c_m) = (\mathcal{A}_i,\mathcal{B}_k,\mathcal{C}_m),
\]
where
\begin{align*}
&\mathcal{A}_i = A_i+a_i\\
&\mathcal{B}_k = B_k+b_k + \frac{1}{2} \sum_{i<j} \alpha_k^{ij}  (A_ia_j-a_iA_j)\\
&\mathcal{C}_m = C_m + c_m +\frac{1}{2} \sum_{i,j} \beta_m^{ij}(A_ib_j -B_ja_i) +\frac{1}{12} \sum_{l,k} \sum_{i<j} (A_l-a_l) \alpha_k^{ij} (A_ia_j-a_iA_j)\beta_{m}^{lk} .
\end{align*}

Observe $(A_i,B_k,C_m)^{-1}= (-A_i,-B_k,-C_m)$ and
\[
(A_i,B_k,C_m)^{-1} \star (a_i,b_k,c_m) = (\tilde{\mathcal{A}_i}, \tilde{\mathcal{B}_k}, \tilde{\mathcal{C}_m}),
\]
where
\begin{align}
&\tilde{\mathcal{A}_i }= a_i-A_i \nonumber \\
&\tilde{\mathcal{B}_k }= b_k-B_k- \frac{1}{2} \sum_{i<j} \alpha_k^{ij}  (A_ia_j-a_iA_j) \label{step-3-ops} \\
&\tilde{\mathcal{C}_m }= c_m-C_m-\frac{1}{2} \sum_{i,j} \beta_m^{ij}(A_ib_j -B_ja_i) +\frac{1}{12} \sum_{l,k}\sum_{i<j} (A_l+a_l) \alpha_k^{ij} (A_ia_j-a_iA_j)\beta_{m}^{lk} . \nonumber
\end{align}

Left-translating  the canonical basis at the origin, we obtain the left-invariant vector fields
\begin{align*}
&X^i = \frac{\partial}{\partial A_i } + \sum_{k=1}^s \frac{1}{2}\left( \sum_{j<i} \alpha_k^{ji}A_j - \sum_{j>i}\alpha_k^{ij}A_j\right)\frac{\partial}{\partial B_k} \\& \ \qquad\qquad+ \sum_{m=1}^t \left[ -\frac{1}{2} \sum_{j=1}^s B_j\beta_m^{ij} +\frac{1}{12} \sum_{l=1}^r \sum_{k=1}^s A_l\left( \sum_{j<i} \alpha_k^{ji} A_j -\sum_{j>i} \alpha_k^{ij} A_j\right) \beta_m^{lk} \right] \frac{\partial }{\partial C_m} ,\\
& Y^k = \frac{d}{\partial B_k} + \sum_{m=1}^t\left(\frac{1}{2} \sum_{i=1}^r\beta_m^{ik} A_i \right)\frac{\partial}{\partial C_m},\\
&Z^m = \frac{\partial}{\partial C_m}. 
\end{align*}
It is clear that $\{X^i\}_i\cup \{Y^k\}_k\cup \{Z^m\}_m$ forms a basis for $Lie(\R^n,\star)$. 
Moreover, we have the expected step three stratification of $Lie(\R^n,\star)$ \cite[Remark 1.4.8]{SLG:BLU}:
\begin{equation}
Lie(\R^{r+s+t},\star) = \langle X^i \rangle_{1\le i \le r}\oplus \langle Y^k\rangle_{1\le k \le s} \oplus \langle Z^m\rangle_{1\le m \le t}\label{step-3-strat}
\end{equation}
In fact, one can show using the Jacobi identity that the linear map $\varphi:Lie(\R^n,\star)\to \mathfrak{g}$ induced by
\[
X^i\mapsto d_i , \quad Y^k\mapsto e_k, \quad Z^m\mapsto f_m
\]
is a Lie algebra isomorphism. 

The contact forms are given by
\begin{align*}
& \omega^k_1 :=dB_k - \sum_{i=1}^r \frac{1}{2}\left(  \sum_{j<i} \alpha_k^{ji} A_i -\sum_{j>i} \alpha_k ^{ij} A_j\right)  dA_i\\
& \omega^m_2:= dC_m - \sum_{i=1}^r \left[ -\frac{1}{2} \sum_{j=1}^s B_j \beta_m^{ij} +\frac{1}{12} \sum_{l=1}^r 
\sum_{k=1}^s A_l (\sum_{j<i} \alpha_k^{ji} A_j - \sum_{j>i} \alpha_k^{ij}A_j )\beta_m^{lk} \right] dA_i.
\end{align*}
We have
\[
H(\R^{r+s+t},\star) = \bigcap_{k=1}^s \ker\omega^k_1 \cap \bigcap_{m=1}^t \ker\omega^m_2,
\]
so that
a tangent vector $v$  lies in $H_p(\R^{r+s+t},\star)$ if and only if $(\omega_1^k)_p (v) = (\omega_2^m)_p(v) =0$
for all $k$ and $m$.



\subsection{Result for step two Carnot groups}\label{step-2-results-section}

In this subsection, we will  consider step two Carnot groups $G$ using the notation from subsection \ref{step-2-setup}.
Recall that we identity $G$ with $\R^{r+s}$ equipped with an operation arising from coordinates of the first kind:
\[
(A_i,B_k)\star (a_i,b_k)= (\mathcal{A}_i,\mathcal{B}_k),
\]
where
\[
\mathcal{A}_i = A_i +a_i, \quad \mathcal{B}_k = B_k+b_k +\frac{1}{2} \sum_{1\le i <j \le r} \alpha_k^{ij}(A_ia_j-a_iA_j).
\]
The contact forms defining the horizontal bundle of $(\R^{r+s},\star)$ are given by
\[
\omega^k = dB_k - \frac{1}{2}\sum_{i=1}^r \left( \sum_{j< i }\alpha_k^{ji} A_j - \sum_{j> i} \alpha_k^{ij}A_j\right) dA_i, \quad k=1,\ldots ,r.
\]
Here, the constants $\alpha_k^{ij} $ come from the bracket relations on $\mathfrak{g}_1$.

The goal of this section will be to prove Theorem \ref{step-2-3-main-theorem} by proving a result similar to Proposition \ref{mf-diff-prop}. 
Theorem \ref{step-2-3-main-theorem} will then follow from this result,  in the same way that Theorem \ref{mf-main-theorem} followed from Proposition \ref{mf-diff-prop}.
We show the following:

\begin{lemma}\label{step-2-hor-lemma}
Fix a step two Carnot group $G$ and an open set $\Omega\subseteq \R^n$. Let
$f=(f^{A_1}, \ldots , f^{A_r},f^{B_1},\ldots , f^{B_s}):\Omega\to G$ be of class $C^{0,\frac{1}{2}+}$, where $f^{A_1},\ldots , f^{A_r}$ are the horizontal components of $f$.  If each  $f^{A_i}$ is differentiable at a point $x_0\in \Omega$,
then $f$ is differentiable at $x_0$ with the image of $df_{x_0}$ contained  in $H_{f(x_0)}G$.
\end{lemma}

\begin{proof}
We need to show for all $k$, the component   $f^{B_k}$ is differentiable at $x_0$ with
\[
df^{B_k}_{x_0} = \frac{1}{2} \sum_{i=1}^r \left( \sum_{j<i}\alpha_k^{ji} f^{A_j}(x_0)- \sum_{j> i} \alpha_k^{ij}f^{A_j}(x_0)\right)df^{A_i}_{x_0}.
\]

Fix $k$. 
By Corollary \ref{carnot-bb-cor}, there exists a constant $C$ such that 
\begin{equation}
|f^{B_k}(x)- f^{B_k}(x_0) - \frac{1}{2}\sum_{1\le i <j \le r} \alpha_k^{ij} (f^{A_i}(x_0) f^{A_j}(x)-f^{A_i}(x)f^{A_j}(x_0))|\le
Cd_{cc}(f(x), f(x_0)) ^2\label{eq-step-2}
\end{equation}
for all $x\in \Omega$.

Choose a function $\beta$ so that (\ref{calpha-ineq}) holds for $f$.
From \eqref{eq-step-2},
\[
|f^{B_k}(x)- f^{B_k}(x_0) -\frac{1}{2} \sum_{1\le i <j \le r} \alpha_k^{ij} (f^{A_i}(x_0) f^{A_j}(x)-f^{A_i}(x)f^{A_j}(x_0))|\le 
\beta^2(|x-x_0|) \cdot |x-x_0|,
\]
absorbing a constant  into $\beta$. 
Thus,
\begin{align*}
&\left|f^{B_k}(x) - f^{B_k}(x_0) - \frac{1}{2} \sum_{i=1}^r \left( \sum_{j<i}\alpha_k^{ji} f^{A_j}(x_0) - \sum_{j> i} \alpha_k^{ij}f^{A_j}(x_0)\right)df^{A_i}_{x_0}(x-x_0)\right|\\
&\qquad  \le \beta^2(|x-x_0|)\cdot |x-x_0|\\
&\qquad \qquad+ \frac{1}{2}\left| \sum_{i<j} \alpha_k^{ij} (f^{A_i}(x_0)f^{A_j}(x) - f^{A_i}(x_0)f^{A_j}(x_0))- 
\alpha_k^{ij} f^{A_i}(x_0) df_{x_0}^{A_j}(x-x_0)\right|\\
&\qquad \qquad +\frac{1}{2}\left| \sum_{i<j} \alpha_k^{ij} (f^{A_i}(x)f^{A_j}(x_0) - f^{A_i}(x_0)f^{A_j}(x_0))- 
\alpha_k^{ij} f^{A_j}(x_0) df_{x_0}^{A_i}(x-x_0)\right|\\
&\qquad =o(|x-x_0|),
\end{align*}
where we used the differentiability of each $f^{A_i}$ at $x_0$ for the last estimate.
\end{proof}

Theorem \ref{step-2-3-main-theorem} for step two Carnot groups now follows from Lemma \ref{step-2-hor-lemma},  using the same reasoning as in the proof of  Theorem \ref{mf-main-theorem}.

\begin{remark}
Lemma \ref{step-2-hor-lemma} was proven in the case $G=\mathbb{H}^n$ by Balogh, Haj{\l}asz, and Wildrick \cite[Proposition 8.1]{WCE:BHW}. 
The proof of Lemma \ref{step-2-hor-lemma} above was directly obtained from their proof by taking into account structural constants.  
\end{remark}


\subsection{Result for step three Carnot groups}

In this section, we will prove Theorem \ref{step-2-3-main-theorem} for step three Carnot groups using similar reasoning as in  subsection \ref{step-2-results-section}. We begin by reviewing notation:

Let $G$ be a step $3$ Carnot group. We identify $G$ with $\R^{r+s+t}$ equipped with an operation arising from coordinates of the first kind:
\[
(A_i,B_k,C_m)\star (a_i,b_k,c_m) = (\mathcal{A}_i,\mathcal{B}_k,\mathcal{C}_m)
\]
where
\[
\mathcal{A}_i = A_i+a_i, \qquad
\mathcal{B}_k = B_k+b_k + \frac{1}{2} \sum_{i<j} \alpha_k^{ij}  (A_ia_j-a_iA_j), 
\]
\[
\mathcal{C}_m = C_m + c_m +\frac{1}{2} \sum_{i,j} \beta_m^{ij}(A_ib_j -B_ja_i) +\frac{1}{12} \sum_{l=1}^r \sum_{k=1}^s \sum_{i<j} (A_l-a_l) \alpha_k^{ij} (A_ia_j-a_iA_j)\beta_{m}^{lk} .
\]
The $1$-forms defining $HG$ are given by 
\begin{align*}
&\omega^k_1 :=dB_k - \sum_{i=1}^r\frac{1}{2} \left(  \sum_{j<i} \alpha_k^{ji} A_i -\sum_{j>i} \alpha_k ^{ij} A_j\right)  dA_i,\\
&\omega^m_2:= dC_m - \sum_{i=1}^r \left[ -\frac{1}{2} \sum_{j=1}^s B_j \beta_m^{ij} +\frac{1}{12} \sum_{l=1}^r 
\sum_{k=1}^s A_l (\sum_{j<i} \alpha_k^{ji} A_j - \sum_{j>i} \alpha_k^{ij}A_j )\beta_m^{lk} \right] dA_i.
\end{align*}

As in Section \ref{step-2-results-section}, to prove Theorem \ref{step-2-3-main-theorem} for step three Carnot groups, it suffices to prove the following:

\begin{lemma}\label{step-3-hor-lemma}
Fix a step three Carnot group $G$ and an open set $\Omega\subseteq \R^n$.
Let $f=(f^{A_1}, \ldots , f^{A_r},f^{B_1},\ldots , f^{B_s}, f^{C_1},\ldots , f^{C_t}):\Omega\to G$ be given of class $C^{0,\frac{1}{2}+}$, where $f^{A_1},\ldots , f^{A_r}$ are the horizontal components of $f$.  If each  $f^{A_i}$ is differentiable at  a point $x_0\in \Omega$,
then $f$ is differentiable at $x_0$ with the image of $df_{x_0} $ lying in $ H_{f(x_0)}G$.
\end{lemma}
 
\begin{proof}
By the proof of Lemma \ref{step-2-hor-lemma}, each component $f^{B_k}$ is differentiable at $x_0$ with
\[
df^{B_k}_{x_0} = \frac{1}{2} \sum_{i} \left( \sum_{j<i}\alpha_k^{ji} f^{A_j}(x_0)- \sum_{j> i} \alpha_k^{ij}f^{A_j}(x_0)\right)df^{A_i}_{x_0}.
\]

It remains to show that  each component $f^{C_m}$ is differentiable at $x_0$ with
\[
df^{C_m}_{x_0}  = \sum_{i}\left[ -\frac{1}{2} \sum_{j} \beta_m^{ij}f^{B_j}(x_0) +\frac{1}{12} \sum_{l,k}  f^{A_l}(x_0)\left(\sum_{j<i} \alpha_k^{ji} f^{A_j}(x_0)- \sum_{j>i} \alpha_k^{ij}f^{A_j}(x_0)\right)\beta_m^{lk} \right] df^{A_i}_{x_0}.
\]

Fix $m$.
Choose   $\beta$ so that (\ref{calpha-ineq}) holds.   
By the calculations in (\ref{step-3-ops}) and Corollary \ref{carnot-bb-cor}, we have 
\begin{align}
&\nonumber |f^{C_m}(x)-f^{C_m}(x_0)-\frac{1}{2} \sum_{i,j} \beta_m^{ij}(f^{A_i}(x_0)f^{B_j}(x) -f^{B_j}(x_0)f^{A_i}(x)) +\frac{1}{12} \sum_{l,k} \sum_{i<j} (f^{A_l}(x_0) \\ &\nonumber \qquad +f^{A_l}(x)) \alpha_k^{ij} (f^{A_i}(x_0)f^{A_j}(x)-f^{A_i}(x)f^{A_j}(x_0))\beta_{m}^{lk} |^{1/3}\le \beta(|x-x_0|) \cdot |x-x_0|^{1/2},\\
&| f^{B_k}(x)-f^{B_k}(x_0)- \frac{1}{2} \sum_{i<j} \alpha_k^{ij}  (f^{A_i}(x_0)f^{A_j}(x)-f^{A_i}(x)f^{A_j}(x_0))|^{1/2} \le \beta(|x-x_0|)\cdot |x-x_0|^{1/2}  \label{step-3-bounds}
\end{align}
for each $k$, absorbing constants into $\beta$.

From (\ref{step-3-bounds}), we have 
\begin{align*}
&\biggl|f^{C_m}(x)-f^{C_m}(x_0)-  \sum_{i} \biggl[ -\frac{1}{2} \sum_{j} \beta_m^{ij}f^{B_j}(x_0)\\ & \qquad +\frac{1}{12} \sum_{l,k}  f^{A_l}(x_0)\biggl(\sum_{j<i} \alpha_k^{ji} f^{A_j}(x_0)- \sum_{j>i} \alpha_k^{ij}f^{A_j}(x_0)\biggr)\beta_m^{lk} \biggr] df^{A_i}_{x_0}(x-x_0)\biggr|\\
&\le \beta(|x-x_0|)^3 \cdot |x-x_0|^{3/2} \\
&+ \biggl| \frac{1}{2} \sum_{i,j} \beta_m^{ij}(f^{A_i}(x_0)f^{B_j}(x) - f^{B_j}(x_0)f^{A_i}(x))-\frac{1}{12} \sum_{l,k}  \sum_{i<j} \biggl[(f^{A_l}(x_0) +f^{A_l}(x)) \cdot \\& \qquad\alpha_k^{ij} (f^{A_i}(x_0)f^{A_j}(x)-f^{A_i}(x)f^{A_j}(x_0))\beta_{m}^{lk}\biggr]+\frac{1}{2}\sum_{i,j} \beta_m^{ij}f^{B_j}(x_0)df^{A_i}_{x_0}(x-x_0)\\ &\qquad-\frac{1}{12} \sum_{l=1}^r    \sum_{k=1}^s f^{A_l}(x_0)\biggl(\sum_{j<i} \alpha_k^{ji} f^{A_j}(x_0)- \sum_{j>i} \alpha_k^{ij}f^{A_j}(x_0)\biggr)\beta_m^{lk} df^{A_i}_{x_0}(x-x_0)\biggr|\\
&\le \beta(|x-x_0|)^3\cdot |x-x_0|^{3/2} \\ &+\left| -\frac{1}{2}\sum_{i,j} \beta_m^{ij} f^{B_j}(x_0)f^{A_i}(x) + \frac{1}{2} \sum_{i,j} \beta_m^{ij} f^{B_j}(x_0)f^{A_i}(x_0) + 
 \frac{1}{2}\sum_{i,j}\beta_m^{ij}f^{B_j}(x_0)df^{A_i}_{x_0}(x-x_0)\right|\\
&+ \biggl|\frac{1}{12} \sum_{l,k}\sum_{i<j} (f^{A_l}(x_0) +f^{A_l}(x)) \alpha_k^{ij}(f^{A_i}(x)f^{A_j}(x_0) -
f^{A_i}(x_0)f^{A_j}(x)) \beta_{m}^{lk} \\& \qquad+ \frac{1}{2}\sum_{i,j} \beta_m^{ij} f^{A_i}(x_0)(f^{B_j}(x) -f^{B_j}(x_0))\\
&\qquad -\frac{1}{12} \sum_{i,l,k} f^{A_l}(x_0) \biggl( \sum_{j<i} \alpha_k^{ji} f^{A_j}(x_0) - \sum_{j>i}
\alpha_k^{ij} f^{A_j}(x_0)\biggr) \beta_m^{lk} df_{x_0}^{A_i} (x-x_0)\biggr|,
\end{align*}
where we regrouped terms in the second inequality. 
As the $f^{A_i}$ are differentiable at $x_0$, we can then  estimate
\begin{align*}
&\biggl|f^{C_m}(x)-f^{C_m}(x_0)-  \sum_{i} \biggl[ -\frac{1}{2} \sum_{j} \beta_m^{ij}f^{B_j}(x_0)\\ & \qquad +\frac{1}{12} \sum_{l,k}  f^{A_l}(x_0)\biggl(\sum_{j<i} \alpha_k^{ji} f^{A_j}(x_0)- \sum_{j>i} \alpha_k^{ij}f^{A_j}(x_0)\biggr)\beta_m^{lk} \biggr] df^{A_i}_{x_0}(x-x_0)\biggr|\\
& \le\beta(|x-x_0|)^3 \cdot |x-x_0|^{3/2} + o(|x-x_0|) \\
&+ \biggl|  -\frac{1}{12} \sum_{l,k}\sum_{i<j} \Al \alpha_k^{ij} \ai \Aj \bmlk 
+\frac{1}{12} \sum_{l,k} \sum_{i<j} \Al \akij \Ai\Aj\bmlk \\
&\qquad + \frac{1}{12} \sum_{l,k}\sum_{i<j} \Al\akij\Aj\bmlk df_{x_0}^{A_i} (x-x_0)\biggr|\\
&+ \biggl| \frac{1}{12} \sum_{l,k} \sum_{j<i}\Al\Aj\akji\ai\bmlk -\frac{1}{12} \sum_{l,k} \sum_{j<i} \Al\Ai\akji\Aj\bmlk \\
&\qquad -\frac{1}{12}\sum_{l,k} \sum_{j<i} \Al\akji\Aj\bmlk df_{x_0}^{A_i}(x-x_0)\biggr|\\
& + \biggl|\frac{1}{6} \sum_{l,k} \sum_{i<j} \Al\akij \Aj\bmlk -\frac{1}{6}\sum_{l,k}\sum_{j<i}\Al\Aj\akji\ai\bmlk
\\ &\qquad + \frac{1}{12}\sum_{l,k} \sum_{i<j} \al\akij (\ai\Aj-\Ai\aj)\bmlk \\&\qquad + \frac{1}{2} \sum_{i,j} \bmij \Ai\bj -\frac{1}{2}\sum_{i,j} \bmij\Bj\Ai \biggr|\\
&\le o(|x-x_0|) \\&+  \biggl|\frac{1}{2} \sum_{l,k} \bmlk \Al\bk - \frac{1}{2} \sum_{l,k} \bmlk\Al\Bk  \\&\qquad + \frac{1}{4} \sum_{l,k} \bmlk\Al \sum_{i<j} \akij 
(\ai\Aj-\aj\Ai)\biggr|\\
&+\biggl| -\frac{1}{12} \sum_{l,k} \bmlk \Al \sum_{i<j}\akij(\ai\Aj-\aj\Ai)\\
&\qquad +\frac{1}{12}\sum_{l,k} \sum_{i<j} \bmlk\al\akij(\ai\Aj-\Ai\aj)\biggr|,
\end{align*}
where we used the differentiability of the $f^{A_i}$ for both inequalities. 
Since the second term in the last expression is bounded by $|x-x_0|\cdot \beta(|x-x_0|)^2$ by  (\ref{step-3-bounds}) (up to a constant factor depending only on $G$),
we can bound the last expression by 
\begin{align*}
&o(|x-x_0|)  + \frac{1}{12} \biggl| \sum_{l,k} \bmlk(\al-\Al )\sum_{i<j} \akij (\ai\Aj-\aj\Ai)\biggr|\\
&\le o(|x-x_0|)+\frac{1}{12} \biggl| \sum_{l,k} \bmlk (\al-\Al ) \sum_{i<j} \akij (\ai\Aj-\Ai\Aj)\biggr|\\
&\qquad+ \frac{1}{12} \biggl| \sum_{l,k} \bmlk (\al-\Al)\sum_{i<j} \akij (\Ai\Aj-\aj\Ai)\biggr|\\
&\le o(|x-x_0|) + \frac{1}{12} \sum_{l,k} \sum_{i<j} |\bmlk\akij \Aj| \cdot |(\al-\Al)(\ai-\Ai)|  \\
&\qquad +\frac{1}{12} \sum_{l,k}\sum_{i<j} |\bmlk\akij \Ai|\cdot | (\al-\Al)(\aj-\Aj)|\\
&=o(|x-x_0|),
\end{align*}
where  for the last equality we noted  
\[
\ai-\Ai= df_{x_0}^{A_i} (x-x_0) +o(|x-x_0|)
=O(|x-x_0|)
\]
for all $i$. 
This proves that each component $f^{C_m}$ is differentiable at $x_0$ with $df^{C_m}_{x_0}$  of the desired form. The 
lemma follows.
\end{proof}

Theorem \ref{step-2-3-main-theorem}  follows for step three Carnot groups from initial remarks.

\section{Future work}

We would like to generalize Theorems \ref{mf-main-theorem} and  \ref{step-2-3-main-theorem} to all Carnot groups. By the work in this paper, it would suffice to prove a result of the form:\\

\noindent \emph{Let $(\R^n,\cdot)$ be a Carnot group and $\Omega\subseteq \R^k$ an open subset. 
Suppose $f:\Omega\to (\R^n,\cdot)$ is of class $C^{0,\frac{1}{2}+}(\Omega,(\R^n,\cdot))$. 
If each of the components of $f$ are differentiable at a point $x_0\in \Omega$, then the image of $df_{x_0}$ lies in $H_{f(x_0)}(\R^n,\cdot)$.}\\


 \noindent Lemmas \ref{step-2-hor-lemma} and  \ref{step-3-hor-lemma} were first proved for model filiform groups and then proved in general by repeating calculations with additional structural constants.
If ones attempts this strategy for higher step Carnot groups, one could run into issues.
For example, there may be a nontrivial bracket relation of elements in the second layer of a general stratification, while such  a relation for a model filiform group must be trivial.
In addition, proving weak contactness (Lemmas \ref{step-2-hor-lemma} and \ref{step-3-hor-lemma}) became much more computationally difficult as one moved from the step two case to the step three case; one would expect this increasing difficulty to continue.
Thus generalizing  these two lemmas may require a  deeper understanding of the polynomials arising from the Baker-Campbell-Hausdorff formula.\\

\noindent \textbf{Acknowledgements.}
The author deeply thanks Jeremy Tyson for bringing this problem to his attention and  for many hours of discussion on the content and presentation of this paper.
The author also thanks Ben Warhurst for explaining why the map $\Phi$ that is used to define coordinates of the second kind for Carnot groups, is a diffeomorphism. 
Finally, we are grateful to the referee for their very careful reading of the paper and for suggesting which results should be cited rather than reproven. 

\bibliographystyle{plain}

\bibliography{variant_gromov_jung}

\end{document}